\documentclass[a4paper,10pt]{article}

\usepackage{geometry}
\usepackage{epsfig}
\usepackage{amsmath}
\usepackage{amssymb}
\usepackage{amsthm}
\usepackage{mathrsfs}
\usepackage{color}
\usepackage{amscd}
\usepackage{euscript}
\usepackage{tikz-cd}

\usepackage[utf8]{inputenc}
\usepackage{fancyhdr}
\usepackage[english]{babel}

\usepackage{graphicx}
\graphicspath{ {./scrivania/} }
\usepackage{afterpage}
\usepackage{authblk}

\theoremstyle{plain} 
\newtheorem{thm}{Theorem}[section] 
\newtheorem{cor}[thm]{Corollary} 
\newtheorem{lem}[thm]{Lemma} 
\newtheorem{prop}[thm]{Proposition} 
\theoremstyle{definition} 
\newtheorem{defn}{Definition}[section] 
\newtheorem{oss}{Remark}

\newtheorem{ex}{Example}

\DeclareFontEncoding{LS1}{}{}
\DeclareFontSubstitution{LS1}{stix}{m}{n}
\DeclareSymbolFont{symbols2}{LS1}{stixfrak} {m} {n}
\DeclareMathSymbol{\operp}{\mathbin}{symbols2}{"A8}
\usepackage[backend=bibtex,style=numeric]{biblatex}
\title{{\bf Partial slice regularity and Fueter's Theorem in several quaternionic variables }}

\author{Giulio Binosi\footnote{ORCID: \texttt{0000-0002-4733-6180}} \\
\small Dipartimento di Matematica, Universit\`a di Trento\\ 
\small Via Sommarive 14, I-38123 Povo Trento, Italy\\
\small giulio.binosi@unitn.it
}
\date{{\small
\textit{
2020 MSC: Primary 30G35; Secondary 30E20; 32A30.
Key words and phrases. Slice-regular functions, Functions of a hypercomplex variable, Quaternions, Clifford algebras.}}}

\begin{document}

\maketitle
\begin{abstract}
    We extend some definitions and give new results about the theory of slice analysis in several quaternionic variables. The sets of slice functions which are respectively slice, slice regular and circular w.r.t. given variables are characterized. We introduce new notions of partial spherical value and derivative for functions of several variables that extend those of one variable. We recover some of their properties as circularity, harmonicity, some relations with differential operators and a Leibniz rule w.r.t. the slice product as well as studying their behavior in the context of several variables. Then, we prove our main result, which is a generalization of Fueter's Theorem for slice regular functions in several variables. This extends the link between slice regular and axially monogenic functions well known in the one variable context.
\end{abstract}

\section{Introduction}
Slice regular functions were firstly introduced in \cite{GentiliStruppa} by Gentili and Struppa for quaternion-valued functions, defined over Euclidean balls with real centre. Exploiting the complex-slice structure of the quaternion algebra $\mathbb{H}$ and following an idea of Cullen \cite{Cullen}, they defined slice regular (or Cullen regular) functions as real differentiable functions which are slice by slice holomorphic.
 The main purpose of this new hypercomplex theory was to overcome the problem encountered by the theory of quaternionic functions already well established by Fueter \cite{Fueter}, in which the class of regular functions does not contain polynomials. 
 On the contrary, the class of slice regular functions contains all the power series with right quaternionic coefficients. The two theories are indeed very skew, since, in general, only constant functions are both Fueter and Cullen regular, even though they present some connections, as Fueter's Theorem suggests.
 We refer the reader to the monograph \cite{LibroCaterina} for a comprehensive treatment of the theory of slice regular functions of one quaternionic variable and to \cite{Gurlebeck}, \cite{Sudbery} for Fueter regular functions.

Interest in this new subject grew rapidly and a large number of papers were published. The theory was soon generalized to more general domains of definition, the so called slice domains \cite{Slicedomain} and extended to octonions \cite{Octonionsetting} and Clifford algebras \cite{Cliffordsetting}. A new viewpoint took place after the work of Ghiloni and Perotti \cite{SRFonAA} with the introduction of stem functions, already used by Fueter to generate axially monogenic functions through Fueter's map \cite{Fueter}. This approach allows to define slice functions, in which no regularity is needed, over any axially symmetric domain and to extend the theory uniformly in any real alternative $^\ast$-algebra with unity.

The stem functions' approach suggested the way to construct a several variable analogue of the theory in the foundational paper \cite{Several}, to which the present article contributes to develop some ideas introduced therein. In that paper, the importance of partial slice regularity has been pointed out. Indeed, it is possible to interpret the slice regularity of an $n$ variables slice function in terms of the one-variable slice regularity of $2^n-1$ slice functions \cite[Theorem 3.23]{Several}, obtained as all possible iterations of partial spherical values and derivatives of that function.
This result establishes a bridge between the one and several variables theories, which has been frequently exploited, for example in \cite{WirtingerPerotti}, where local slice analysis was naturally extended from one to several quaternionic variables. But, the study of partial slice regularity, as well as partial spherical values and derivatives was not developed further and a more detailed study deserved attention, leading to this work.


We describe the structure of the paper. After briefly recalling the theory of slice regular functions of one and several quaternionic variables, we focus on the study of partial slice properties, i.e. sliceness, slice regularity or circularity
w.r.t. a specific subset of variables (\S 3). More precisely, given a set of variables $\{x_h\}_{h\in H}$, we characterize (Propositions \ref{characterization sliceness in xh}, \ref{Proposition characterization h slice regularity} and \ref{proposition characterization partial circularity}) the sets $\mathcal{S}_H$, $\mathcal{S}\mathcal{R}_H$ and $\mathcal{S}_{c,H}$ of slice functions which are, respectively, slice, slice regular and circular w.r.t. all the variables $x_h$. 
The use of stem functions is fundamental as all those characterizations are given through conditions over stem functions. Furthermore,  we show that for every choice of $H\in\mathcal{P}(n)$, the set $\mathcal{S}_{c,H}$ forms a subalgebra of the set of slice functions endowed with the slice product $(\mathcal{S},\odot)$ (Corollary \ref{Corollary subalgebra}); $\mathcal{S}_H$ and $\mathcal{S}\mathcal{R}_H$ do not share this property.

In Chapter 4, we define partial spherical values and derivatives for functions of several variables, which extend the one-variable analogues. We recover some of their main properties such as harmonicity (Proposition \ref{proposizione derivata sferica e armonica}), representation and Leibniz formulas (\ref{representation formula several}), (\ref{Leibniz rule}) and we find new ones, peculiar of the several variables setting (Proposition \ref{proposizione proprieta derivata sferica}), through the characterizations of Chapter 3. Finally, thanks to the harmonicity of the partial spherical derivatives, we prove a generalization of Fueter's Theorem for slice regular functions of several quaternionic variables (Theorem \ref{Teorema di Fueter in piu variabili}), which extends the link between slice regular and axially monogenic functions in higher dimensions.


\section{Preliminaries}
We briefly recall the main definitions of the theory of slice regular functions of one and several quaternionic variables. We state here the definitions of \cite{SRFonAA} and \cite{Several}, reduced to the quaternionic setting.
\subsection{Slice regular functions of one quaternionic variable}
\label{subsection one variable theory}
Let $\mathbb{H}$ denote the algebra of quaternions with basis elements $\{1,i,j,k\}$. We can embed $\mathbb{R}\subset\mathbb{H}$ as the subalgebra generated by $1$, while $\operatorname{Im}(\mathbb{H}):=<i,j,k>$, whence $\mathbb{H}=\mathbb{R}\oplus\operatorname{Im}(\mathbb{H})$. Let $\mathbb{S}_\mathbb{H}:=\{q\in\mathbb{H}\mid q^2=-1\}\subset\operatorname{Im}(\mathbb{H})$ be the sphere of square roots of $-1$, then if $q\in\mathbb{H}\setminus\mathbb{R}$, there exist $\alpha,\beta\in\mathbb{R}$, $J\in\mathbb{S}_\mathbb{H}$ such that $q=\alpha+J\beta$. They are unique if we require $\beta>0$. Every such $q$ generates a sphere we denote with $\mathbb{S}_q=\mathbb{S}_{\alpha,\beta}=\{\alpha+I\beta: I\in\mathbb{S}_\mathbb{H}\}$. Given $J\in\mathbb{S}_\mathbb{H}$, let $\phi_J:\mathbb{C}\ni \alpha+i\beta\mapsto\alpha+J\beta\in\mathbb{H}$. It is clear \cite[(1)]{SRFonAA} that $\phi_J$ is a real $^\ast$-algebras isomorphism onto $\mathbb{C}_J:=\left<1,J\right>_\mathbb{R}\subset\mathbb{H}$.

Denote with $\{1,e_1\}$ a basis of $\mathbb{R}^2$.
Let $D\subset\mathbb{C}$ be a conjugate invariant domain ($\overline{D}=D$), a function $F:D\rightarrow\mathbb{H}\otimes\mathbb{R}^2$ is a stem function if it is complex intrinsic, i.e. $F(\overline{z})=\overline{F(z)}$, which means that if $F$ has components $F=F_\emptyset+e_1F_1$, they satisfy $F_\emptyset(\overline{z})=F_\emptyset(z)$ and $F_1(\overline{z})=-F_1(z)$. Given such a set $D$, we define its circularization in $\mathbb{H}$ as $\Omega_D:=\{\alpha+J\beta\mid \alpha+i\beta\in D, J\in\mathbb{S}_\mathbb{H}\}=\bigcup_{\alpha+i\beta\in D}\mathbb{S}_{\alpha,\beta}$. We can associate to every stem function $F=F_\emptyset+e_1F_1:D\rightarrow\mathbb{H}\otimes\mathbb{R}^2$ a unique slice function $f=\mathcal{I}(F):\Omega_D\rightarrow\mathbb{H}$ as follows: if $x=\alpha+J\beta=\phi_J(z)$ for some $z=\alpha+i\beta\in D$ and $J\in\mathbb{S}_\mathbb{H}$, we define
\begin{equation*}
    f(x)=F_\emptyset(z)+JF_1(z).
\end{equation*}
Every slice function can be completely recovered by its value over one slice $\mathbb{C}_J$, with a representation formula \cite[Proposition 6]{SRFonAA}: let $I,J\in\mathbb{S}_\mathbb{H}$, then for every $x=\alpha+I\beta$ it holds
\begin{equation}
\label{representation formula}
    f(x)=\frac{1}{2}f((\alpha+J\beta)+f(\alpha-J\beta))-\frac{IJ}{2}(f(\alpha+J\beta)-f(\alpha+J\beta)).
\end{equation}
Given a slice function $f$, we define its spherical value and its spherical derivative respectively as 
\begin{equation*}
\label{definition spherical operators one variable}
    f^\circ_s(x)=\frac{1}{2}(f(x)+f(\overline{x})),\qquad f'_s(x)=\frac{1}{2}[\operatorname{Im}(x)]^{-1}(f(x)-f(\overline{x})).
\end{equation*}
Note that the spherical value and the spherical derivative are both slice functions, as $f^\circ_s=\mathcal{I}(F_\emptyset)$ and $f'_s=\mathcal{I}(F_1(z)/\operatorname{Im}(z))$, if $f=\mathcal{I}(F_\emptyset+e_1F_1)$. Moreover, applying (\ref{representation formula}) with $I=J$ we get
\begin{equation*}
    f(x)=f^\circ_s(x)+\operatorname{Im}(x)f'_s(x).
\end{equation*}
We can define a product over slice functions. Let $F,G$ be two stem functions with $F=F_\emptyset+e_1 F_1$ and $G=G_\emptyset+e_1G_1$. Define $F\otimes G:=F_\emptyset G_\emptyset-F_1G_1+e_1(F_\emptyset G_1+F_1G_\emptyset)$, which happens to be a stem function. Now, if $f=\mathcal{I}(F)$ and $g=\mathcal{I}(G)$, define $f\odot g:=\mathcal{I}(F\otimes G)$. With respect to this product, the spherical derivative satisfies a Lebniz rule:
\begin{equation*}
\label{leibniz formula one variable}
    (f\odot g)'_s=f'_s\odot g^\circ_s+f^\circ_s\odot g'_s.
\end{equation*}
Let $F$ be a $\mathcal{C}^1$ stem function. Define 
\begin{equation*}
    \frac{\partial F}{\partial z}=\frac{1}{2}\left(\frac{\partial F}{\partial\alpha}-e_1 \frac{\partial F}{\partial\beta}\right),\qquad \frac{\partial F}{\partial \overline{z}}=\frac{1}{2}\left(\frac{\partial F}{\partial\alpha}+e_1 \frac{\partial F}{\partial\beta}\right).
\end{equation*}
Since both $\partial F/\partial z$ and $\partial F/\partial \overline{z}$ are stem functions, we can define
\begin{equation*}
    \frac{\partial f}{\partial x}=\mathcal{I}\left(\frac{\partial F}{\partial z}\right),\qquad \frac{\partial f}{\partial x^c}=\mathcal{I}\left(\frac{\partial F}{\partial \overline{z}}\right).
\end{equation*}
Finally, a slice function $f=\mathcal{I}(F)$ is said to be slice regular if $\partial f/\partial x^c=0$ or, equivalently, if $\partial F/\partial \overline{z}=0$. Note that \cite[Proposition 8]{SRFonAA}, if $\Omega_D\cap\mathbb{R}\neq\emptyset$, the definition of slice regular function coincide with the one given by Gentili and Struppa \cite{GentiliStruppa}, namely that, for every $J\in\mathbb{S}_\mathbb{H}$ the restriction of $f$, $f_J:\Omega_D\cap\mathbb{C}_J\rightarrow\mathbb{H}$ is holomorphic w.r.t. the complex structure defined by multiplication by $J$.

\subsection{Slice regular functions of several quaternionic variables}

Let $n$ be a positive integer and let $\mathcal{P}(n)$ denote all possible subsets of $\{1,...,n\}$. Given an ordered set $K=\{k_1,...,k_p\}\in\mathcal{P}(n)$, with $k_1<...<k_p$ and an associated $p$-tuple $(q_{k_1},...,q_{k_p})\in\mathbb{H}^p$, we define $q_K:=q_{k_1}\cdot...\cdot q_{k_p}$ (with $q_\emptyset:=1$) and for any $\tilde{q}\in\mathbb{H}$, $[q_K,\tilde{q}]:=q_K\cdot\tilde{q}$.

Given $z=(z_1,...,z_n)\in\mathbb{C}^n$, set $\overline{z}^h:=(z_1,...,z_{h-1},\overline{z}_h,z_{h+1},...,z_n)$, $\forall h=1,...,n$. A set $D\subset\mathbb{C}^n$ is called invariant w.r.t. complex conjugation whenever $z\in D$ if and only if $\overline{z}^h\in D$ for every $h\in\{1,...,n\}.$ We define its circularization $\Omega_D\subset\mathbb{H}^n$ as
\begin{equation*}
    \Omega_D:=\{(\alpha_1+J_1\beta_1,...,\alpha_n+J_n\beta_n)\mid (\alpha_1+i\beta_1,...,\alpha_n+i\beta_n)\in D, J_1,...,J_n\in\mathbb{S}_\mathbb{H}\}
\end{equation*}
and we call circular those sets $\Omega$ such that $\Omega=\Omega_D$ for some $D\subset\mathbb{C}^n$, invariant w.r.t. complex conjugation.
From now on, we will always assume $D$ an invariant subset of $\mathbb{C}^n$ w.r.t. complex conjugation and $\Omega_D$ a circular set of $\mathbb{H}^n$.

Let $\{e_1,...,e_n\}$ be an orthonormal frame of $\mathbb{R}^n$ and denote with $\{e_K\}_{K\in\mathcal{P}(n)}$ a basis of $\mathbb{R}^{2^n}$. Consider a function $F:D\rightarrow\mathbb{H}\otimes\mathbb{R}^{2^n}$, $F=\sum_{K\in\mathcal{P}(n)}e_KF_K$,
in which its components $\left\{F_K\right\}_{K\in\mathcal{P}(n)}$ are $\mathbb{H}$-valued functions. We call $F$ a stem function if $\forall K\in\mathcal{P}(n)$, $\forall h=1,...,n$
\begin{equation}
\label{defining property of stem}
    F_K(\overline{z}^h)=(-1)^{|K\cap\{h\}|}F_K(z).
\end{equation}
Write $Stem(D)$ for the set of all stem functions from $D$ to $\mathbb{H}\otimes\mathbb{R}^{2^n}$.

A map $f:\Omega_D\subset\mathbb{H}^n\rightarrow\mathbb{H}$ is called slice function if there exists a stem function $F:D\rightarrow\mathbb{H}\otimes\mathbb{R}^{2^n}$, $F=\sum_{K\in\mathcal{P}(n)}e_KF_K$, such that
\begin{equation*}
    f(x)= \sum_{K\in\mathcal{P}(n)}[J_K,F_K(z)],\qquad\forall x\in\Omega_D,
\end{equation*}
where $x=(x_1,...,x_n)$, with $x_i=\alpha_i+J_i\beta_i$, for some $\alpha_i,\beta_i\in\mathbb{R}$, $J_i\in\mathbb{S}_\mathbb{H}$ and $z=(z_1,...,z_n)\in D$, $z_i=\alpha_i+i\beta_i$, for $i=1,...,n$. Note that (\ref{defining property of stem}) is necessary to make slice functions well defined. We say that $f$ is induced by $F$. $\mathcal{S}(\Omega_D)$ will denote the set of all slice functions from $\Omega_D$ to $\mathbb{H}$ and $\mathcal{I}:Stem(D)\rightarrow \mathcal{S}(\Omega_D)$ will be the map sending a stem function to its induced slice function. By \cite[Proposition 2.12]{Several}, every slice function is induced by a unique stem function, so $\mathcal{I}$ is an injective map.

We can define slice functions through a commutative diagram too: for any $J_1,...,J_n\in\mathbb{S}_\mathbb{H}$ define $$\phi_{J_1}\times...\times\phi_{J_n}:\mathbb{C}^n\ni(z_1,...,z_n)\mapsto\left(\phi_{J_1}(z_1),...,\phi_{J_n}(z_n)\right)\in\mathbb{H}^n$$
and $$\Phi_{J_1,...,J_n}:\mathbb{H}\otimes\mathbb{R}^{2^n}\ni \sum_{K\in\mathcal{P}(n)}e_Ka_K\mapsto \sum_{K\in\mathcal{P}(n)}\left[J_K,a_K\right]\in\mathbb{H}.$$
Given $F\in Stem(D)$, we can define its induced slice function $f=\mathcal{I}(F)$ as the unique slice function that makes the following diagram commutative for any $J_1,...,J_n\in\mathbb{S}_\mathbb{H}$:
\begin{center}
    \begin{tikzcd}
	{D} && {\mathbb{H}\otimes\mathbb{R}^{2^n}} \\
	& \circlearrowleft \\
	{\Omega_D} && {\mathbb{H}}
	\arrow["F", from=1-1, to=1-3]
	\arrow["{\phi_{J_1}\times...\times\phi_{J_n}}"', from=1-1, to=3-1]
	\arrow["f"', from=3-1, to=3-3]
	\arrow["{\Phi_{J_1,...,J_n}}", from=1-3, to=3-3]
\end{tikzcd}
\end{center}

As described in \cite[Definition 2.31, Lemma 2.32]{Several}, equip $\mathbb{R}^{2^n}$ with a $\Delta$-product  $\otimes:\mathbb{R}^{2^n}\times\mathbb{R}^{2^n}\rightarrow\mathbb{R}^{2^n}$, defined on each basis element as
\begin{equation*}
    e_H\otimes e_K:=(-1)^{|H\cap K|}e_{H\Delta K},
\end{equation*}
where $H\Delta K=(H\cup K)\setminus(H\cap K)$ and extended by linearity to all $\mathbb{R}^{2^n}$. This product induces a product on $\mathbb{H}\otimes\mathbb{R}^{2^n}$: given $a,b\in\mathbb{H}\otimes\mathbb{R}^{2^n}$, $a=\sum_{H\in\mathcal{P}(n)}e_Ha_H$ and $b=\sum_{K\in\mathcal{P}(n)}e_Kb_K$, with $a_H,b_K\in\mathbb{H}$, define
\begin{equation*}
    a\otimes b:= \sum_{H,K\in\mathcal{P}(n)}(e_H\otimes e_K)(a_Hb_K)= \sum_{H,K\in\mathcal{P}(n)}(-1)^{|H\cap K|}e_{H\Delta K}a_Hb_K,
\end{equation*}
where $a_Hb_K$ is just the usual product of quaternions.
Furthermore, we can define a product between stem functions as the pointwise product induced by $\otimes$: let $F,G\in Stem(D)$, define $(F\otimes G)(z):=F(z)\otimes G(z)$. More precisely, if $F=\sum_{H\in\mathcal{P}(n)}e_HF_H$ and $G=\sum_{K\in\mathcal{P}(n)}e_KG_K$,
\begin{equation*}
    (F\otimes G)(z):= \sum_{H,K\in\mathcal{P}(n)}(-1)^{|H\cap K|}e_{H\Delta K}F_H(z)G_K(z).
\end{equation*}
The advantage of this definition is that the product of two stem functions is again a stem function \cite[Lemma 2.34]{Several} and this allows to define a product on slice functions, too. Let $f,g\in\mathcal{S}(\Omega_D)$, with $f=\mathcal{I}(F)$ and $g=\mathcal{I}(G)$, then define the slice tensor product $f\odot g$ between $f$ and $g$ as 
\begin{equation*}
    f\odot g:=\mathcal{I}(F\otimes G).
\end{equation*}

Equip $\mathbb{R}^{2^n}$ with the family of commutative complex structures $\mathcal{J}=\left\{\mathcal{J}_h:\mathbb{R}^{2^n}\rightarrow\mathbb{R}^{2^n}\right\}_{h=1}^n,$
where each $\mathcal{J}_h$ is defined over any basis element $e_K$ of $\mathbb{R}^{2^n}$ as
\begin{equation*}
    \mathcal{J}_h(e_K):=(-1)^{|K\cap\{h\}|}e_{K\Delta\{h\}}=\left\{\begin{array}{ll}
        e_{K\cup\{h\}} &\text{ if }h\notin K  \\
        -e_{K\setminus\{h\}} &\text{ if }h\in K,
    \end{array}\right.
\end{equation*}
and extend it by linearity to all $\mathbb{R}^{2^n}$. $\mathcal{J}$ induces a family of commutative complex structure on $\mathbb{H}\otimes\mathbb{R}^{2^n}$ (by abuse of notation, we use the same symbol) 
$\mathcal{J}=\left\{\mathcal{J}_h:\mathbb{H}\otimes\mathbb{R}^{2^n}\rightarrow\mathbb{H}\otimes\mathbb{R}^{2^n}\right\}_{h=1}^n$ according to the formula $$\mathcal{J}_h(q\otimes a):=q\otimes\mathcal{J}_h(a)\qquad\forall q\in\mathbb{H},\quad\forall a\in\mathbb{R}^{2^n}.$$

We can associate two Cauchy-Riemann operators to each complex structure $\mathcal{J}_h$. Given $F\in Stem(D)\cap\mathcal{C}^1(D)$, we define
\begin{equation*}
    \partial_hF:=\dfrac{1}{2}\left(\dfrac{\partial F}{\partial \alpha_h}-\mathcal{J}_h\left(\dfrac{\partial F}{\partial\beta_h}\right)\right),\qquad \overline{\partial}_hF:=\dfrac{1}{2}\left(\dfrac{\partial F}{\partial \alpha_h}+\mathcal{J}_h\left(\dfrac{\partial F}{\partial\beta_h}\right)\right).
\end{equation*}
Note that, if $F$ is a stem function, so are $\partial_hF$ and $\overline{\partial}_hF$ \cite[Lemma 3.9]{Several}. Thus, if $f=\mathcal{I}(F)\in\mathcal{S}^1(\Omega_D):=\mathcal{I}(Stem(D)\cap\mathcal{C}^1(\Omega_D))$, we can define the partial derivatives for every $h=1,...,n$
\begin{equation*}
    \dfrac{\partial f}{\partial x_h}:=\mathcal{I}\left(\partial_hF\right),\qquad \dfrac{\partial f}{\partial  x^c_h}:=\mathcal{I}\left(\overline{\partial}_hF\right).
\end{equation*}
A $\mathcal{C}^1$ stem function $F=\sum_{K\in\mathcal{P}(n)}e_KF_K$ is called $h$-holomorphic w.r.t. $\mathcal{J}$ if $\overline{\partial}_hF\equiv0$ or equivalently \cite[Lemma 3.12]{Several}, if its components satisfies a system of Cauchy-Riemann equations
\begin{equation}
    \label{Cauchy-Riemann equations}
    \dfrac{\partial F_K}{\partial\alpha_h}=    \dfrac{\partial F_{K\cup\{h\}}}{\partial\beta_h},\qquad
    \dfrac{\partial F_K}{\partial\beta_h}=-    \dfrac{\partial F_{K\cup\{h\}}}{\partial\alpha_h},\qquad\forall K\in\mathcal{P}(n), h\notin K
\end{equation}
and it is called holomorphic if it is $h$-holomorphic for every $h=1,...,n$. Finally, given a holomorphic stem function $F$, the induced slice function $\mathcal{I}(F)$ will be called slice regular function. The set of all slice regular functions from $\Omega_D$ to $\mathbb{H}$ will be denoted by $\mathcal{S}\mathcal{R}(\Omega_D)$. By \cite[Proposition 3.13]{Several}, $f\in\mathcal{S}\mathcal{R}(\Omega_D)$ if and only if $\frac{\partial f}{\partial  x^c_h}=0$ for every $h=1,...,n$.

We recall two other operators on $\mathbb{H}$, known as Cauchy-Riemann-Fueter operators:
\begin{equation*}
    \partial_{CRF}:=\dfrac{\partial}{\partial\alpha}-i\dfrac{\partial}{\partial\beta}-j\dfrac{\partial}{\partial\gamma}-k\dfrac{\partial}{\partial\delta},\qquad
    \overline{\partial}_{CRF}:=\dfrac{\partial}{\partial\alpha}+i\dfrac{\partial}{\partial\beta}+j\dfrac{\partial}{\partial\gamma}+k\dfrac{\partial}{\partial\delta},
\end{equation*}
where $\alpha$, $\beta$, $\gamma$ and $\delta$ denotes the four real components of a quaternion $x=\alpha+i\beta+j\gamma+k\delta$. Functions in the kernel of $\overline{\partial}_{CRF}$ are usually called Fueter regular (or monogenic in the context of Clifford algebras). The importance of these operators 
is evident as they
factorize the Laplacian, indeed
\begin{equation*}
\label{equazione fattorizzazione laplaciano}
    \partial_{CRF}\overline{\partial}_{CRF}=\overline{\partial}_{CRF}\partial_{CRF}=\Delta,
\end{equation*}
thus, monogenic functions are in particular harmonic.
We can extend these operators to $\mathbb{H}^n$: for a slice function  $f:\Omega_D\rightarrow\mathbb{H}$, we define, for any $h=1,...,n$, $\partial_{x_h}$ and $\overline{\partial}_{x_h}$ as the Cauchy-Riemann-Fueter operators w.r.t. $x_h:=\alpha_h+i\beta_h+j\gamma_h+k\delta_h$:
\begin{equation*}
    \partial_{x_h}:=\dfrac{\partial}{\partial\alpha_h}-i\dfrac{\partial}{\partial\beta_h}-j\dfrac{\partial}{\partial\gamma_h}-k\dfrac{\partial}{\partial\delta_h},\qquad
    \overline{\partial}_{x_h}:=\dfrac{\partial}{\partial\alpha_h}+i\dfrac{\partial}{\partial\beta_h}+j\dfrac{\partial}{\partial\gamma_h}+k\dfrac{\partial}{\partial\delta_h}.
\end{equation*}
For every $h=1,...,n$ it holds
\begin{equation*}
\label{equazione fattorizzazione laplaciano}
    \partial_{x_h}\overline{\partial}_{x_h}=\overline{\partial}_{x_h}\partial_{x_h}=\Delta_h,
\end{equation*}
where $\Delta_h=\frac{\partial^2}{\partial\alpha^2_h}+\frac{\partial^2}{\partial\beta^2_h}+\frac{\partial^2}{\partial\gamma^2_h}+\frac{\partial^2}{\partial\delta^2_h}$.
Finally, denote by $\mathcal{M}_h(\Omega):=\{f:\Omega\rightarrow\mathbb{H}:\overline{\partial}_{x_h}f=0\}$ the set of monogenic functions w.r.t $x_h$ and let $\mathcal{A}\mathcal{M}_h(\Omega_D):=\mathcal{M}_h(\Omega_D)\cap\mathcal{S}^1(\Omega_D)$ be the set of axially monogenic functions w.r.t. $x_h$, i.e. the set of slice functions which are monogenic w.r.t. $x_h$.


\section{Characterization of $\mathcal{S}_H$, $\mathcal{S}\mathcal{R}_H$ and $\mathcal{S}_{c,H}$}
\label{Sezione proprietà sliceness rispetto a variabile}
Let $f:\Omega_D\subset\mathbb{H}^n\rightarrow\mathbb{H}$ and $h=1,...,n$. For any $y=(y_1,...,y_n)\in\Omega_D$, let 
$$\Omega_{D,h}(y):=\{x\in\mathbb{H}\mid (y_1,...,y_{h-1},x,y_{h+1},...,y_n)\in\Omega_D\}\subset\mathbb{H}.$$
It is easy to see (\cite[\S 2]{Several}) that $\Omega_{D,h}(y)$ is a circular set of $\mathbb{H}$, more precisely $\Omega_{D,h}(y)=\Omega_{D_h(z)}$, where
\begin{equation*}
    D_h(z):=\{w\in\mathbb{C}\mid (z_1,...,z_{h-1},w,z_{h+1},...,z_n)\in D\},
\end{equation*}
and $z=(z_1,...,z_n)$ is such that $y\in\Omega_{\{z\}}$. 

\begin{defn}
We say that a slice function $f\in\mathcal{S}(\Omega_D)$ is \emph{slice} (resp. \emph{slice regular}) \emph{w.r.t.} $x_h$ if, $\forall y\in\Omega_D$, its restriction
\begin{equation*}
    f^y_h:\Omega_{D,h}(y)\rightarrow\mathbb{H}, \ f^y_h(x):=f(y_1,...,y_{h-1},x,y_{h+1},...,y_n)
\end{equation*}
is a one variable slice (resp. slice regular) function, as defined in \S\ref{subsection one variable theory} . 
We denote by $\mathcal{S}_h(\Omega_D)$ (resp. $\mathcal{S}\mathcal{R}_h(\Omega_D)$) the set of slice functions from $\Omega_D$ to $\mathbb{H}$ that are slice (resp. slice regular) w.r.t. $x_h$. For $H\in\mathcal{P}(n)$, define $\mathcal{S}_H(\Omega_D):=\bigcap_{h\in H}\mathcal{S}_h(\Omega_D)$, $\mathcal{S}\mathcal{R}_H(\Omega_D):=\bigcap_{h\in H}\mathcal{S}\mathcal{R}_h(\Omega_D)$. Note that, by definition, $\mathcal{S}\mathcal{R}_H(\Omega_D)\subset\mathcal{S}_H(\Omega_D)\subset\mathcal{S}(\Omega_D)$.
 
We say that $f$ is \emph{circular w.r.t.} $x_h$ if $\forall y=(y_1,...,y_n)\in\Omega_D$, $f^y_h$ is constant on $\mathbb{S}_{y_h}\subset\mathbb{H}$.
The set of slice functions which are circular w.r.t. $x_h$ will be denoted by $\mathcal{S}_{c,h}(\Omega_D)\subset\mathcal{S}(\Omega_D)$. Note that $f$ is circular w.r.t. $x_h$ if and only if for every orthogonal transformation $T:\mathbb{H}\rightarrow\mathbb{H}$ that fixes $1$, it holds $f(x_1,...,x_{h-1},T(x_h),x_{h+1},...,x_n)=f(x_1,...,x_n)$, for any $(x_1,...,x_n)\in\Omega_D$. In this case, if $x_h=\alpha_h+J_h\beta_h$, $f$ does not depend on $J_h$. Finally, if $H\in\mathcal{P}(n)$, set $\mathcal{S}_{c,H}(\Omega_D):=\bigcap_{h\in H}\mathcal{S}_{c,h}(\Omega_D)$.
\end{defn}

Every slice function is, in particular, slice w.r.t. the first variable \cite[Proposition 2.23]{Several}, i.e. $\mathcal{S}_1(\Omega_D)=\mathcal{S}(\Omega_D)$, but in general $\mathcal{S}_h(\Omega_D)\subsetneq \mathcal{S}(\Omega_D)$. The next proposition characterizes the set $\mathcal{S}_H(\Omega_D)$ for any $H\in\mathcal{P}(n)$ in terms of stem functions.

\begin{prop}
\label{characterization sliceness in xh}
For every $H\in\mathcal{P}(n)$ it holds
\begin{equation}
\label{equazione caratterizzazione sliceness in H}
    \mathcal{S}_H(\Omega_D)=\left\{\mathcal{I}(F) : F\in Stem(D), \ F=\sum_{K\in H^c}e_HF_K+ \sum_{h\in H}e_{\{h\}}\sum_{Q\subset\{h+1,...,n\}\setminus H}e_QF_{\{h\}\cup Q}\right\}.
\end{equation}
In particular, for any $h\in\{1,...,n\}$, 
\begin{equation}
\label{equation characterization sliceness in xh}
    \mathcal{S}_h(\Omega_D)=\left\{\mathcal{I}(F) : F\in Stem(D), \ F=\sum_{K\in\mathcal{P}(n), h\notin K}e_HF_K+e_{\{h\}}\sum_{Q\subset\{h+1,...,n\}}e_QF_{\{h\}\cup Q}\right\}.
\end{equation}
Equivalently, $f=\mathcal{I}(F)\in\mathcal{S}_H(\Omega_D)$ if and only if $F_{P\cup\{h\}\cup Q}=0$, $\forall h\in H$, $\forall Q\subset\{h+1,...,n\}$, $\forall P\in\mathcal{P}(h-1)$ with $P\neq\emptyset$.
\end{prop}

\begin{proof}
Since $\mathcal{S}_H(\Omega_D):=\bigcap_{h\in H}\mathcal{S}_h(\Omega_D)$, it is sufficient to assume $H=\{h\}$ for some $h=1,...,n$.
\begin{enumerate}
    \item [$\Rightarrow)$] $f\in\mathcal{S}_h(\Omega_D)$ means that $\forall y\in\Omega_D$, the one-variable function $f^y_h$ is slice, thus, it must satisfies representation formula (\ref{representation formula}): namely, if $x=a+Ib\in\Omega_{D,h}(y)$ and $J\in\mathbb{S}_\mathbb{H}$, it holds
    \begin{equation}
    \label{representation formula for fyh}
        f^y_h(x)=\dfrac{1}{2}\left(f^y_h(a+Jb)+f^y_h(a-Jb)\right)-\dfrac{IJ}{2}\left(f^y_h(a+Jb)-f^y_h(a-Jb)\right).
    \end{equation}
    Set $z=(z_1,...,z_n)$, $z'=(z_1,...,z_{h-1})$, $z"=(z_{h+1},...,z_n)$, $y=(\phi_{J_1}\times...\times\phi_{J_n})(z)$, for some $J_1,...,J_n\in\mathbb{S}_\mathbb{H}$, $w=a+ib$, $x=\phi_I(w)$, $L_s=M_s=J_s$ for $s\neq h$, $L_h:=I$ and $M_h:=J$. Then we have
    \begin{equation}
    \label{lk unita immaginarie}
        \begin{split}
            f^y_h(x)=\sum_{K\in\mathcal{P}(n), h\notin K}[J_K,F_K(z',w,z'')]+ \sum_{K\in\mathcal{P}(n), h\notin K}[L_{K\cup\{h\}},F_{K\cup\{h\}}(z',w,z'')],
        \end{split}
    \end{equation}
    \begin{equation*}
    \begin{split}
        f^y_h(a+Jb)&= \sum_{K\in\mathcal{P}(n),h\notin K}[J_K,F_K(z',w,z'')]+ \sum_{K\in\mathcal{P}(n),h\notin K}[M_{K\cup\{h\}},F_{K\cup\{h\}}(z',w,z'')]
    \end{split}
    \end{equation*}
    and
    \begin{equation*}
    \begin{split}
        f^y_h(a-Jb)&= \sum_{K\in\mathcal{P}(n),h\notin K}[J_K,F_K(z',\overline{w},z'')]+ \sum_{K\in\mathcal{P}(n),h\notin K}[M_{K\cup\{h\}},F_{K\cup\{h\}}(z',\overline{w},z'')]\\
        &= \sum_{K\in\mathcal{P}(n),h\notin K}[J_K,F_K(z',w,z'')]- \sum_{K\in\mathcal{P}(n),h\notin K}[M_{K\cup\{h\}},F_{K\cup\{h\}}(z',w,z'')],
    \end{split}
    \end{equation*}
    where we have used (\ref{defining property of stem}). Thus, the right hand side of (\ref{representation formula for fyh}) becomes
    \begin{equation}
    \label{rhs riscritto}
        \begin{split}
            &\dfrac{1}{2}\left(f^y_h(a+Jb)+f^y_h(a-Jb)\right)-\dfrac{I}{2}\left[J\left(f^y_h(a+Jb)-f^y_h(a-Jb)\right)\right]=\\
            &= \sum_{K\in\mathcal{P}(n),h\notin K}[J_K,F_K(z',w,z'')]-IJ \sum_{K\in\mathcal{P}(n),h\notin K}[M_{K\cup\{h\}},F_{K\cup\{h\}}(z',w,z'')].
        \end{split}
    \end{equation}
    Comparing (\ref{lk unita immaginarie}) and (\ref{rhs riscritto}), (\ref{representation formula for fyh}) is satisfied if and only if
    \begin{equation}
    \label{equazione con sommatoria}
         \sum_{K\in\mathcal{P}(n), h\notin K}[L_{K\cup\{h\}},F_{K\cup\{h\}}(z',w,z'')]=-IJ \sum_{K\in\mathcal{P}(n),h\notin K}[M_{K\cup\{h\}},F_{K\cup\{h\}}(z',w,z'')].
    \end{equation}
    Since (\ref{representation formula for fyh}) is assumed to be true for every $I,J,J_1,...,J_n\in\mathbb{S}_\mathbb{H}$ and every $z',w,z''$, (\ref{equazione con sommatoria}) holds if and only if $\forall K\subset\{1,...,n\}\setminus\{h\}$
    \begin{equation}
     \label{eq senza somma}
        [L_{K\cup\{h\}},F_{K\cup\{h\}}(z',w,z'')]=-IJ[M_{K\cup\{h\}},F_{K\cup\{h\}}(z',w,z'')].
    \end{equation}
Indeed, if (\ref{eq senza somma}) were not true, there would be a $K\subset\mathcal{P}(\{1,...,n\}\setminus \{h\})$ such that
\begin{equation*}
    [L_{K\cup\{h\}},F_{K\cup\{h\}}(z',w,z'')]\neq-IJ[M_{K\cup\{h\}},F_{K\cup\{h\}}(z',w,z'')],
\end{equation*}
but for $J_1=...=J_n=J=I$ we would have
\begin{equation*}
    (-1)^{|K\cup\{h\}|}F_{K\cup\{h\}}(z',w,z'')\neq(-1)^{|K\cup\{h\}|}F_{K\cup\{h\}}(z',w,z''),
\end{equation*}
which is false. Let us represent $\{K\in\mathcal{P}(n)\mid h\notin K\}=\{P\sqcup Q\mid P\in\mathcal{P}(h-1), \ Q\subset\{h+1,...,n\}\}$. 
    Suppose $P\neq\emptyset$, then $\forall Q\subset\{h+1,...,n\})$, (\ref{eq senza somma}) becomes
    \begin{equation*}
        [L_{(P\cup\{h\}\cup Q)},F_{P\cup\{h\}\cup Q}(z',w,z'')]=-IJ[M_{(P\cup\{h\}\cup Q)},F_{P\cup\{h\}\cup Q}(z',w,z'')]
    \end{equation*}
    and this implies that $F_{P\cup\{h\}\cup Q}\equiv0$.
    Indeed, if $F_{P\cup\{h\}\cup Q}\neq0$, the previous equation would reduce to $J_PI=-IJJ_PJ$ which does not hold for every choice of $I,J,J_P$.
    \item[$\Leftarrow)$] Vice versa, suppose $F$ takes the form
    \begin{equation*}
        F= \sum_{K\in\mathcal{P}(n),h\notin K}e_KF_K+e_h \sum_{Q\subset\{h+1,...,n\}}e_QF_{\{h\}\cup Q}.
    \end{equation*}
    Following the notation above, it holds
    \begin{equation*}
        f^y_h(x)= \sum_{K\in\mathcal{P}(n),h\notin K}[J_K,F_K(z',w,z'')]+I \sum_{Q\subset\{h+1,...,n\}}[J_Q,F_{\{h\}\cup Q}(z',w,z'')].
    \end{equation*}
    Thus, consider the function $G^y_h=G^y_{1,h}+iG^y_{2,h}$, with
    \begin{equation*}
        G^y_{1,h}(w):= \sum_{K\in\mathcal{P}(n),h\notin K}[J_K,F_K(z',w,z'')]
    ,\quad
        G^y_{2,h}(w):= \sum_{Q\subset\{h+1,...,n\}}[J_Q,F_{\{h\}\cup Q}(z',w,z'')].
    \end{equation*}
    $G^y_h$ is a one-variable stem function, indeed,
    \begin{equation*}
    \begin{split}
        G^y_h(\overline{w})&= \sum_{K\in\mathcal{P}(n),h\notin K}[J_K,F_K(z',\overline{w},z'')]+i \sum_{Q\subset\{h+1,...,n\}}[J_Q,F_{\{h\}\cup Q}(z',\overline{w},z'')]\\
        &= \sum_{K\in\mathcal{P}(n),h\notin K}[J_K,F_K(z',w,z'')]-i \sum_{Q\subset\{h+1,...,n\}}[J_Q,F_{\{h\}\cup Q}(z',w,z'')]=\overline{G^y_h(w)}
    \end{split}
    \end{equation*}
    and $f^y_h=\mathcal{I}(G^y_h)$, by construction, so $f\in\mathcal{S}_h(\Omega_D)$.
\end{enumerate}
\end{proof}

\begin{oss}
\label{remark structure slice function wrt xh}
By the previous proof, we can better understand the set $\mathcal{S}_H(\Omega_D)$: let $f=\mathcal{I}(F)\in\mathcal{S}_H(\Omega_D)$, then for any $x\in\Omega_D$ with $x=(\phi_{J_1}\times...\times \phi_{J_n})(z)$, $f(x)$ takes the form
\begin{equation*}
    f(x)= \sum_{K\in H^c}\left[J_K,F_K(z)\right]+ \sum_{h\in H}J_h \sum_{Q\subset\{h+1,...,n\}\setminus H}\left[J_Q,F_{\{h\}\cup Q}(z)\right].
\end{equation*}
Moreover, for any $h\in H$ and any $y=(y_1,...,y_n)$, $f^y_h$ is a one-variable slice function, induced by the stem function $G^y_h$, with components
\begin{equation}
\label{Equazione componenti stem parziale}
        G^y_{1,h}(w):= \sum_{K\in\mathcal{P}(n),h\notin K}[J_K,F_K(z',w,z'')]
    ,\qquad
        G^y_{2,h}(w):= \sum_{Q\subset\{h+1,...,n\}}[J_Q,F_{\{h\}\cup Q}(z',w,z'')],
    \end{equation}
    where $z=(z',z_h,z")$ and $y=(\phi_{J_1}\times...\times\phi_{J_n})(z)$.
\end{oss}

Now, we deal with partial slice regularity.
\begin{prop}
\label{Proposition characterization h slice regularity}
For every $H\in\mathcal{P}(n)$ it holds
\begin{equation*}
    \mathcal{S}\mathcal{R}_H(\Omega_D)=\mathcal{S}_H(\Omega_D)\bigcap_{h\in H}\ker(\partial/\partial x_h^c).
\end{equation*}
\end{prop}
\begin{proof}
Since $\mathcal{S}\mathcal{R}_H(\Omega_D):=\bigcap_{h\in H}\mathcal{S}\mathcal{R}_h(\Omega_D)$, it is sufficient to assume $H=\{h\}$ for some $h=1,...,n$.
\begin{enumerate}
        \item [$\subset)$] By definition, $\mathcal{S}\mathcal{R}_h(\Omega_D)\subset\mathcal{S}_h(\Omega_D)$, so let $f=\mathcal{I}(F)$, with
\begin{equation}
\label{form of stem xh slice}
    F= \sum_{K\in\mathcal{P}(n),h\notin K}e_KF_K+e_h \sum_{Q\subset\{h+1,...,n\}}e_QF_{\{h\}\cup Q},
\end{equation}
thanks to \eqref{equation characterization sliceness in xh}.
For any $y\in\Omega_D$, $f^y_h$ is induced by the stem function $G^y_h=G^y_{1,h}+iG^y_{2,h}$, with
\begin{equation*}
        G^y_{1,h}(w):= \sum_{K\in\mathcal{P}(n),h\notin K}[J_K,F_K(z',w,z'')]
,\quad
        G^y_{2,h}(w):= \sum_{Q\subset\{h+1,...,n\}}[J_Q,F_{\{h\}\cup Q}(z',w,z'')].
\end{equation*}
         By definition, $f\in\mathcal{S}\mathcal{R}_h(\Omega_D)$ means that $\forall y\in\Omega_D$, the stem function $G^y_h$ is holomorphic, i.e. recalling \eqref{Equazione componenti stem parziale} it must hold that for every $z=(z',z_h,z")\in D$, $w\in D_h(z)$ and $\forall J_j\in\mathbb{S}_\mathbb{H}$ that
        \begin{equation*}
            \left\{\begin{array}{l}
             \sum_{P,Q}[J_{P\cup Q},\partial_{\alpha_h}F_{P\cup Q}(z',w,z'')]= \sum_{Q}[J_Q,\partial_{\beta_h}F_{\{h\}\cup Q}(z',w,z'')]  \\
               \sum_{P,Q}[J_{P\cup Q},\partial_{\beta_h}F_{P\cup Q}(z',w,z'')]=- \sum_{Q}[J_Q,\partial_{\alpha_h}F_{\{h\}\cup Q}(z',w,z'')],
        \end{array}\right.
        \end{equation*}
        where in the above sums $P\in\mathcal{P}(h-1)$ and $Q\subset\{h+1,...,n\}$.
        Now, since that system is true for every choice of imaginary unit $J_j$, proceeding as in the proof of Proposition \ref{characterization sliceness in xh} we can deduce that an equivalence between each term of the sum holds. Let any $Q\subset\{h+1,...,n\}$: if $P\neq\emptyset$, equality can sussist only if $\partial_{\alpha_h}F_{P\cup Q}=\partial_{\beta_h}F_{P\cup Q}=0$ and this trivially proves that the components $F_{P\cup Q}$ satisfies (\ref{Cauchy-Riemann equations}), since $F_{P\cup\{h\}\cup Q}=0$, by (\ref{equation characterization sliceness in xh}). Otherwise, let $P=\emptyset$, then the previous system becomes
        \begin{equation*}
            \left\{\begin{array}{l}
                \partial_{\alpha_h}F_Q=\partial_{\beta_h}F_{\{h\}\cup Q}  \\
                  \partial_{\beta_h}F_Q=-\partial_{\alpha_h}F_{\{h\}\cup Q}
            \end{array}\right.
        \end{equation*}
and (\ref{Cauchy-Riemann equations}) are satisfied too. This proves that $F$ is $h$-holomorphic, which means that $f\in\ker(\partial/\partial x_h^c)$.
\item[$\supset)$] Suppose $f\in\mathcal{S}_h(\Omega_D)\cap\ker(\partial /\partial x_h^c)$, then $F$ satisfies (\ref{form of stem xh slice}) and (\ref{Cauchy-Riemann equations}).
         As in the proof of Proposition \ref{characterization sliceness in xh}, represent $K=P\sqcup Q$, with $P\in\mathcal{P}(h-1)$ and $Q\subset\{h+1,...,n\}$.
         Since, by (\ref{form of stem xh slice}), $F_{P\cup\{h\}\cup Q}\equiv0$, $\forall P\in\mathcal{P}(h-1)\setminus\{\emptyset\}$, $\forall Q\subset\{h+1,...,n\}$ the $h$-holomorphicity of $F$ reduces to the following conditions:
        \begin{equation}
            \label{system h holomorphicity}
            \left\{\begin{array}{l}
                \partial_{\alpha_h}F_{P\cup Q}=\partial_{\beta_h}F_{P\cup Q}=0 \\
                \partial_{\alpha_h}F_Q=\partial_{\beta_h}F_{\{h\}\cup Q}  \\
                  \partial_{\beta_h}F_Q=\partial_{\alpha_h}F_{\{h\}\cup Q}.
            \end{array}\right.
        \end{equation}
        On the other hand, $f\in\mathcal{S}\mathcal{R}_h(\Omega_D)$ if and only if $G^y_h$ is a slice regular function $\forall y\in\Omega_D$, which means that $\partial_\alpha G^y_{1,h}=\partial_\beta G^y_{2,h}$ and $\partial_\beta G^y_{1,h}=-\partial_\alpha G^y_{2,h}$, which, by definition of $G^y_h$ is equivalent to
        \begin{equation*}
            \left\{\begin{array}{l}
            \partial_{\alpha_h} \sum_{K\in\mathcal{P}(n),h\notin K}[J_K,F_K(z)]=\partial_{\beta_h} \sum_{Q\subset\{h+1,...,n\}}[J_Q,F_{\{h\}\cup Q}(z)]  \\
              \partial_{\beta_h} \sum_{K\in\mathcal{P}(n),h\notin K}[J_K,F_K(z)]=-\partial_{\alpha_h} \sum_{Q\subset\{h+1,...,n\}}[J_Q,F_{\{h\}\cup Q}(z)],
        \end{array}\right.
        \end{equation*}
        where $y=(\phi_{J_1}\times...\times\phi_{J_n})(z)$, $z=(z_1,...,z_n)$, $z_j=\alpha_j+i\beta_j$.
        Let us prove the first row of the system. Using the first two equation of (\ref{system h holomorphicity}) and splitting $K=P\sqcup Q$, we can write the left side as
        \begin{equation*}
            \begin{split}
                &\partial_{\alpha_h} \sum_{P\in\mathcal{P}(h-1),Q\subset\{h+1,...,n\}}[J_{P\cup Q},F_{P\cup Q}(z',w,z'')]\\
                &= \sum_{P\in\mathcal{P}(h-1),Q\subset\{h+1,...,n\}}[J_{P\cup Q},\partial_{\alpha_h}F_{P\cup Q}(z',w,z'')]= \sum_{Q\subset\{h+1,...,n\}}[J_Q,\partial_{\alpha_h}F_Q(z',w,z'')]\\
                &= \sum_{Q\subset\{h+1,...,n\}}[J_Q,\partial_{\beta_h}F_{\{h\}\cup Q}(z',w,z'')]=\partial_{\beta_h} \sum_{Q\subset\{h+1,...,n\}}[J_Q,F_{\{h\}\cup Q}(z',w,z'')].
            \end{split}
        \end{equation*}
        The second equation is proved in the same way.
    \end{enumerate}
    \end{proof}
    
    
\begin{cor}
Let $f\in\mathcal{S}\mathcal{R}(\Omega_D)$ and $H\in\mathcal{P}(n)$. Then $f\in\mathcal{S}_H(\Omega_D)$ if and only if $ f\in\mathcal{S}\mathcal{R}_H(\Omega_D).$
\end{cor}
\begin{proof}
The "if" part is trivial. Viceversa, note that by \cite[Proposition 3.13]{Several}, $f\in\mathcal{S}\mathcal{R}(\Omega_D)$, implies $\partial f/\partial x_h^c=0$, $\forall h=1,...,n$, hence $\mathcal{S}_H(\Omega_D)\cap\mathcal{S}\mathcal{R}(\Omega_D)\subset\mathcal{S}_H(\Omega_D)\bigcap_{h\in H}\ker(\partial/\partial x_h^c)=\mathcal{S}\mathcal{R}_H(\Omega_D)$, by Proposition \ref{Proposition characterization h slice regularity}.
\end{proof}

Finally, we characterize circularity.
\begin{prop}
\label{proposition characterization partial circularity}
For every $H\in\mathcal{P}(n)$ it holds
\begin{equation}
\label{equazione circolarita}
    \mathcal{S}_{c,H}(\Omega_D)=\left\{\mathcal{I}(F): F\in Stem(D), F= \sum_{K\subset H^c}e_KF_K\right\}.
\end{equation}
In particular, $\mathcal{S}_{c,H}(\Omega_D)\subset\mathcal{S}_H(\Omega_D)$.
\end{prop}

\begin{proof}
Since $\mathcal{S}_{c,H}(\Omega_D)=\bigcap_{h\in H}\mathcal{S}_{c,h}(\Omega_D)$, it is sufficient to assume $H=\{h\}$ for some $h=1,...,n$.
Let any $y=(y_1,...,y_n)\in\Omega_D$, with $y_j:=\alpha_j+J_j\beta_j$, $z_j:=\alpha_j+i\beta_j$, set $z'=(z_1,...,z_{h-1})$ and $z"=(z_{h+1},...,z_n)$. $f\in\mathcal{S}_{c,h}(\Omega_D)$ if for every $x=a+Ib$, $f^y_h(x)$ does not depend on $I$. Let $w:=a+ib$, $M_p:=J_p$ if $p\neq h$ and $M_h=I$, then
    \begin{equation*}
        f^y_h(x)= \sum_{K\in\mathcal{P}(n),h\notin K}[J_K,F_K(z',w,z")]+ \sum_{K\in\mathcal{P}(n),h\notin K}[M_{K\cup\{h\}},F_{K\cup\{h\}}(z',w,z")].
    \end{equation*}
    It is clear that $f^y_h(a+Ib)$ does not depend on $I$ if and only if $F_{K\cup\{h\}}=0$ for every $K\in\mathcal{P}(n)$. Finally, comparing (\ref{equazione caratterizzazione sliceness in H}) and (\ref{equazione circolarita}) we see that $\mathcal{S}_{c,H}(\Omega_D)\subset\mathcal{S}_H(\Omega_D)$.
\end{proof}
Note that functions of the form \eqref{equazione circolarita} were introduced in \cite{Several} as $H^c$-reduced slice functions, hence we can say that $f\in\mathcal{S}_{c,H}(\Omega_D)$ if and only if it is $H^c$-reduced. It is easy now to prove the following property.

\begin{cor}
\label{Corollary subalgebra}
    For every $H\in\mathcal{P}(n)$, the set $\mathcal{S}_{c,H}(\Omega_D)$ is a real subalgebra of $(\mathcal{S}(\Omega_D),\odot)$.
\end{cor}
\begin{proof}
    We need to show that if $f,g\in\mathcal{S}_{c,H}(\Omega_D)$, then $f\odot g\in\mathcal{S}_{c,H}(\Omega_D)$. Let $f=\mathcal{I}(F)$ and $g\in\mathcal{I}(G)$, with $F=\sum_{K\subset H^c}e_KF_K$ and $G=\sum_{T\subset H^c}e_TG_T$, by \eqref{equazione circolarita}. Then
    \begin{equation*}
        F\otimes G=\sum_{K,T\subset H^c}(-1)^{|K\cap T|}e_{K\Delta T}F_KG_T,
    \end{equation*}
    with $K\Delta T=(K\cup T)\setminus (K\cap T)\subset K\cup T\subset H^c$. Then, again \eqref{equazione circolarita} implies $f\odot g\in\mathcal{S}_{c,H}(\Omega_D)$.
\end{proof}

Note that the previous result does not apply to $\mathcal{S}_H(\Omega_D)$, nor $\mathcal{S}\mathcal{R}_H(\Omega_D)$, unless for $\mathcal{S}_1(\Omega_D)=\mathcal{S}(\Omega_D)$ and $\mathcal{S}\mathcal{R}_1(\Omega_D)$. Indeed, for example, $x_1,x_2\in\mathcal{S}\mathcal{R}_2(\Omega_D)$, while $x_1\odot x_2\notin \mathcal{S}_2(\Omega_D)$.

Slice regularity and circularity are hardly compatible.

\begin{prop}
Let $f\in\mathcal{S}_{c,h}(\Omega_D)\cap\mathcal{S}\mathcal{R}_h(\Omega_D)$. Then $f$ is locally constant w.r.t. $x_h$.
\end{prop}

\begin{proof}
Let $x_h=a_h+J_hb_h$ and $f=\mathcal{I}(F)$. Since $f\in\mathcal{S}_{c,h}(\Omega_D)$, $f$ does not depend on $J_h$ and $F_{K\cup\{h\}}=0$ for any $K\in\mathcal{P}(n)$. Moreover, $f\in\mathcal{S}\mathcal{R}_h(\Omega_D)\subset\ker(\partial/\partial x_h^c)$, by Proposition \ref{Proposition characterization h slice regularity}, so by (\ref{Cauchy-Riemann equations})
\begin{equation*}
    \dfrac{\partial F_K}{\partial\alpha_h}=\dfrac{\partial F_{K\cup\{h\}}}{\partial\beta_h}=0=\dfrac{\partial F_{K\cup\{h\}}}{\partial\alpha_h}=-\dfrac{\partial F_K}{\partial\beta_h}.
\end{equation*}
Thus, $f$ does not depend neither on $\alpha_h$ and $\beta_h$ and so it is locally constant w.r.t. $x_h$.
\end{proof}

\begin{ex}
Consider the following polynomial function
$f:\mathbb{H}^3\rightarrow\mathbb{H},\  f(x_1,x_2,x_3):=x_1x_3+x_2x_3^2k,$
which happens to be a slice regular function, \cite[Proposition 3.14]{Several}. We claim that $f\in\mathcal{S}\mathcal{R}_2(\Omega_D)$.
Let us explicit the components of the stem function inducing $f$: let $z=(z_1,z_2,z_3)\in\mathbb{C}^3$, with $z_j:=\alpha_j+i\beta_j$, then $f=\mathcal{I}(F)$, with $F=\sum_{K\in\mathcal{P}(3)}e_KF_K$, where
\begin{align*}
&F_\emptyset(z)=\alpha_1\alpha_3+\alpha_2(\alpha_3^2-\beta_3^2)k,\qquad F_{\{1\}}(z)=\beta_1\alpha_3,\qquad F_{\{2\}}(z)=\beta_2(\alpha_3^2-\beta_3^2)k,\\
&F_{\{3\}}(z)=\alpha_1\beta_3+2\alpha_2\alpha_3\beta_3k,\quad F_{\{1,2\}}(z)=0,\qquad F_{\{1,3\}}(z)=\beta_1\beta_3,\qquad F_{\{2,3\}}(z)=2\beta_2\alpha_3\beta_3k,\\
& F_{\{1,2,3\}}(z)=0.
\end{align*}    
Thus, $F$ has the structure required by (\ref{equation characterization sliceness in xh}) for $h=2$, so $f\in\mathcal{S}_2(\Omega_D)$. Moreover, for $K=\emptyset,\{1\},\{3\},\{1,3\}$ it holds $$\frac{\partial F_K}{\partial \alpha_2}=\frac{\partial F_{K\cup\{2\}}}{\partial \beta_2},\qquad\frac{\partial F_K}{\partial \beta_2}=-\frac{\partial F_{K\cup\{2\}}}{\partial \alpha_2},$$
so $f\in\ker(\partial/\partial x_2^c)$ and so $f\in\mathcal{S}\mathcal{R}_2(\Omega_D)=\mathcal{S}_2(\Omega_D)\cap\ker(\partial/\partial x_2^c)$.

We could have proven the claim by definition, through Remark \ref{remark structure slice function wrt xh}, which explicitly gives us the stem function that induces the corresponding one variable slice function, for every choice of $y$. Fix any $y=(y_1,y_2,y_3)\in\mathbb{H}^3$, then $f^y_2$ is a slice regular function, induced by the holomorphic stem function $G^y_2=G^y_{1,2}+iG^y_{2,2}$, with
\begin{equation*}
    G^y_{1,2}(\alpha+i\beta)=y_1y_3+\alpha y_3^2k,\qquad G^y_{2,2}(\alpha+i\beta)=\beta y_3^2k.
\end{equation*}
\end{ex}

\section{Partial spherical derivatives}
For $h\in\{1,...,n\}$, define $\mathbb{R}_h:=\{(x_1,...,x_n)\mid x_h\in\mathbb{R}\}$ and for $H\in\mathcal{P}(n)$, $\mathbb{R}_H:=\bigcup_{h\in H}\mathbb{R}_h$.
\begin{defn}
    Let $F:D\subset\mathbb{C}^n\rightarrow\mathbb{H}\otimes\mathbb{R}^{2^n}$ be a stem function. Define for $h=1,...,n$ and for $H=\{h_1,...,h_p\}\in\mathcal{P}(n)$
\begin{equation*}
        \begin{split}
            F^\circ_h(z):&= \sum_{K\in\mathcal{P}(n),h\notin K}e_KF_K(z),\\
        F^\circ_H(z):&=\sum_{K\subset H^c}e_KF_K(z)=\left(\dots(F^\circ_{h_1})^\circ_{h_2}\dots\right)^\circ_{h_p}(z)
        \end{split}
    \end{equation*}
    and 
    \begin{align}
        F'_h(z):&=\beta_h^{-1} \sum_{K\in\mathcal{P}(n),h\notin K}e_KF_{K\cup\{h\}}(z),&\text{if }z\in D\setminus\mathbb{R}_h\\
        F'_H(z):&=\beta_H^{-1}\sum_{K\in H^c}e_KF_{K\cup H}(z)=\left(\dots(F'_{h_1})'_{h_2}\dots\right)'_{h_p}(z), &\text{if }z\in D\setminus\mathbb{R}_H,
    \end{align}
    where $z=(z_1,...,z_n)$ with $z_j=\alpha_j+i\beta_j$
    and $\beta_H=\prod_{h\in H}\beta_h$.
    
    \end{defn}

    
\begin{lem}
 For every $H\in\mathcal{P}(n)$, $F^\circ_H$ and $F'_H$ are well defined stem functions on $D$ and $D\setminus\mathbb{R}_H$, respectively.
\end{lem}

\begin{proof}
Firstly, let us prove that $F^\circ_H$ and $F'_H$ are well defined, i.e. their definition does not depend on the order of $H$. Indeed, for any $i,j=1,...,n$ it holds
    \begin{equation*}
        (F'_i)'_j(z)= \sum_{K\in\mathcal{P}(n),i,j\notin K}e_K\beta^{-1}_j\beta_i^{-1}F_{K\cup\{i,j\}}(z)=(F'_j)'_i(z)
    \end{equation*}
    and analogously for $(F^\circ_i)^\circ_j$.
Without loss of generality, assume $H=\{h\}$, for some $h=1,...,n$. 
$F^\circ_h$ is trivially a stem function because its non zero components are the same of $F$. Let us explicit $F'_h=\sum_{K\in\mathcal{P}(n)}e_KG_K$, with
\begin{equation*}
    G_K(z)=\left\{\begin{array}{ll}
         \beta_h^{-1}F_{K\cup\{h\}}\qquad&\text{ if }h\notin K  \\
          0 \qquad&\text{ if }h\in K,
    \end{array}\right.
\end{equation*}
we will show that every component of $F'_h$ satisfies \eqref{defining property of stem}. Let us consider only the components $G_K$, with $h\notin K$, otherwise \eqref{defining property of stem} is trivial. For any $m\neq h$ we have
\begin{equation*}
    G_K(\overline{z}^m)= 
    \beta_h^{-1}F_{K\cup\{h\}}(\overline{z}^m)=\beta_h^{-1}(-1)^{|K\cap\{m\}|}F_{K\cup\{h\}}(z)=(-1)^{|K\cap\{m\}|}G_K(z),
\end{equation*}
while, for $m=h$
\begin{equation*}
    G_K(\overline{z}^h)=
    (-\beta_h^{-1})F_{K\cup\{h\}}(\overline{z}^h)=(-\beta_h^{-1})(-F_{K\cup\{h\}}(z))= \beta_h^{-1}F_{K\cup\{h\}}(z)=G_K(z).
\end{equation*}
\end{proof}
The previous Lemma allows to make the following

\begin{defn}
    Let $f=\mathcal{I}(F)\in\mathcal{S}(\Omega_D)$. For $h\in\{1,...,n\}$, we define its \emph{spherical $x_h$-value and $x_h$-derivative} rispectively as
    \begin{equation*}
       f^\circ_{s, h}:=\mathcal{I}(F^\circ_{h}), \qquad f'_{s, h}:=\mathcal{I}(F'_h).
    \end{equation*}
Analogously, for $H\in\mathcal{P}(n)$, define
    \begin{equation*}
        f^\circ_{s,H}:=\mathcal{I}(F^\circ_H),\qquad f'_{s,H}:=\mathcal{I}(F'_H).
    \end{equation*}
Note that $f^\circ_{s,H}\in\mathcal{S}(\Omega_D)$, while $f'_{s,H}\in\mathcal{S}(\Omega_{D_H})$, where $\Omega_{D_H}:=\Omega_D\setminus\mathbb{R}_H$.
\end{defn}

We stress that the terms spherical value and spherical derivatives have been already used in \cite[\S 2.3]{Several} in the context of slice functions of several quaternionic variables, but they refer to different objects. With respect to our definition, spherical values and derivatives are more related to the truncated spherical derivatives $\mathcal{D}_\epsilon(f)$ \cite[Definition 2.24]{Several}, where for $h\in\{1,...,n\}$ and $\epsilon:\{1,...,h\}\to\{0,1\}$, $\mathcal{D}_\epsilon(f):=\mathcal{D}_{x_h}^{\epsilon(h)}\cdots\mathcal{D}_{x_1}^{\epsilon(1)}(f)$, with $\mathcal{D}_{x_l}^{1}(f)=f'_{s,l}$ and $\mathcal{D}_{x_l}^{0}(f)=f'_{s,l}$. Indeed, it holds $\mathcal{D}_\epsilon(f)=(f'_{s,H})^\circ_{s,K}$, with $H=\epsilon^{-1}(1)$ and $K=\epsilon^{-1}(0)$.

The following proposition justifies the names given to $f^\circ_{s,h}$ and $f'_{s,h}$, comparing them to their one-variable analogues (\S\ref{definition spherical operators one variable}). Note that we have to assume $f\in\mathcal{S}_h(\Omega_D)$, in order for the spherical derivative to agree with it.

\begin{prop}
Let $f\in\mathcal{S}(\Omega_D)$ and $h=1,...,n$. Then it holds
\begin{enumerate}
    \item $\forall x=(x_1,...,x_n)\in\Omega_D$
    \begin{equation*}
        f^\circ_{s, h}(x)=\dfrac{1}{2}\left(f(x)+f\left(\overline{x}^h\right)\right)=(f^x_h)^\circ_s(x_h);
    \end{equation*}
    \item if $f\in\mathcal{S}_h(\Omega_D)$,  $\forall x\in\Omega_{D}\setminus\mathbb{R}_h$  
    \begin{equation}
    \label{equazione derivata sferica parziale coincide con unidimensionale}
        f'_{s, h}(x)=\left[2\operatorname{Im}(x_h)\right]^{-1}(f(x)-f(\overline{x}^h))=(f^x_h)'_{s}(x_h).
    \end{equation}
        In particular, if we assume $f\in\mathcal{S}^1(\Omega_D)$, then we can extend the definition of $f'_{s,h}$ to all $\Omega_D$, thanks to \cite[Proposition 7, (2)]{SRFonAA}.

\end{enumerate} 
\end{prop}

\begin{proof}
Let $f=\mathcal{I}(F)$, with $F=\sum_{K\in\mathcal{P}(n)}e_KF_K$. Then for any $z\in D$ and $x=(\phi_{J_1}\times...\times\phi_{J_n})(z)$ we get
\begin{equation*}
    \begin{split}
&f(x)+f(\overline{x}^h)= \sum_{K\in\mathcal{P}(n)}\left([J_K,F_K(z)]+[J_K,F_K(\overline{z}^h)]\right)\\
&= \sum_{K\in\mathcal{P}(n)}\left([J_K,F_K(z)]+(-1)^{|K\cap\{h\}|}[J_K,F_K(z)]\right)= \sum_{K\in\mathcal{P}(n),h\notin K}\left(2[J_K,F_K(z)]\right)=2f^\circ_{s, h}(x).
    \end{split}
\end{equation*}
Now, assume $f\in\mathcal{S}_h(\Omega_D)$, then by (\ref{equation characterization sliceness in xh}) $$f(x)= \sum_{h\notin K}\left[J_K,F_K(z)\right]+J_h \sum_{Q\subset\{h+1,...,n\}}\left[J_Q,F_{\{h\}\cup Q}(z)\right]$$ and so $$f'_{s, h}(x)=\sum_{Q\subset\{h+1,...,n\}}[J_Q,\beta_h^{-1}F_{\{h\}\cup Q}(z)].$$ On the other hand, let $x=(\phi_{J_1}\times...\times\phi_{J_n})(z)$, then by (\ref{defining property of stem}) we have
\begin{equation*}
    \begin{split}
        f(x)-f\left(\overline{x}^h\right)&= \sum_{K\in\mathcal{P}(n),h\notin K}[J_K,F_K(z)]+J_h\sum_{Q\subset\{h+1,...,n\}}[J_Q,F_{\{h\}\cup Q}(z)]+\\
        &\qquad- \sum_{K\in\mathcal{P}(n),h\notin K}[J_K,F_K(\overline{z}^h)]-J_h\sum_{Q\subset\{h+1,...,n\}}[J_Q,F_{\{h\}\cup Q}(\overline{z}^h)]\\
        &=2J_h\sum_{Q\subset\{h+1,...,n\}}[J_Q,F_{\{h\}\cup Q}(z)],
    \end{split}
\end{equation*}
from which
\begin{equation*}
    \begin{split}
        &\left[2\operatorname{Im}(x_h)\right]^{-1}\left(f(x)-f\left(\overline{x}^h\right)\right)=\left[2J_h\beta_h\right]^{-1}\left(2J_h\sum_{Q\subset\{h+1,...,n\}}[J_Q,F_{\{h\}\cup Q}(z)]\right)\\
        &= \sum_{Q\subset\{h+1,...,n\}}[J_Q,\beta_h^{-1}F_{\{h\}\cup Q}(z)]=f'_{s,h}(x).
    \end{split}
\end{equation*}
\end{proof}
We extend from \cite{Harmonicity} properties of the spherical derivative of one-variable slice regular functions to several variables.

\begin{lem}
\label{lemma derivata sferica operatore crf e Laplaciano}
If $f\in\mathcal{S}\mathcal{R}_h(\Omega_D)$, the following hold:
\begin{enumerate}
    \item $\overline{\partial}_{x_h}f=-2f'_{s,h}$;
    \item $\Delta_h f=-4\dfrac{\partial f'_{s,h}}{\partial x_h}=-2\partial_{x_h}(f'_{s,h})$.
\end{enumerate}
\end{lem}

\begin{proof}
\begin{enumerate}
    \item Note that $\forall y=(y_1,...,y_n)\in\Omega_D$, $f^y_h\in\mathcal{S}\mathcal{R}(\Omega_{D,h}(y))$, then we can apply (\ref{equazione derivata sferica parziale coincide con unidimensionale}) and \cite[Corollary 6.2, (a)]{Harmonicity} to get
\begin{equation*}
 \overline{\partial}_{x_h}f(y)=\overline{\partial}_{CRF}(f^y_h)(y_h)=-2(f^y_h)'_s(y_h)=-2f'_{s,h}(y).
\end{equation*}
\item By (\ref{equazione derivata sferica parziale coincide con unidimensionale}), \cite[Corollary 6.2, (c), Theorem 6.3, (c)]{Harmonicity} and \cite[Theorem 2.2 (ii)]{Global} we have
\begin{equation*}
    \begin{split}
        \Delta_hf(y)=\Delta(f^y_h)(y_h)&=-4\dfrac{\partial (f^y_h)'_s}{\partial x}(y_h)=-4\theta(f^y_h)'_s(y_h)=-2\partial_{CRF}(f^y_h)'_s(y_h)=-2\partial_{x_h}f'_{s,h}(y),
    \end{split}
\end{equation*}
where $(\theta f)(x)=\frac{1}{2}\left(\frac{\partial f}{\partial \alpha}(x)+\frac{\operatorname{Im}(x)}{|\operatorname{Im}(x)|^2}(\beta\frac{\partial f}{\partial\beta}(x)+\gamma\frac{\partial f}{\partial\gamma}(x)+\delta\frac{\partial f}{\partial\delta}(x))\right)$ satisfies $\theta f=\frac{\partial f}{\partial x}$ and $2\theta f'_s=\partial_{CRF}f'_s$  for any slice function $f$.\footnote{In \cite{Harmonicity} the factor 1/2 is omitted in the definition of $\theta$.}
\end{enumerate}
\end{proof}

The next proposition presents some properties of partial spherical values and derivatives peculiar of the several variables setting.
\begin{prop}
\label{proposizione proprieta derivata sferica}
Let $f\in\mathcal{S}(\Omega_D)$, $h\in\{1,...,n\}$ and $H\in\mathcal{P}(n)$, with $p=\min H^c$ if $H\neq\{1,...,n\}$. Then
\begin{enumerate}
    \item $f^\circ_{s,H}\in\mathcal{S}_{c,H}(\Omega_D)\cap\mathcal{S}_p(\Omega_D)$ and $f'_{s,H}\in\mathcal{S}_{c,H}(\Omega_{D_H})\cap\mathcal{S}_p(\Omega_{D_H})$;
    \item if $f\in\mathcal{S}_h(\Omega_{D})$, $f'_{s, h}\in\mathcal{S}_{h+1}(\Omega_{D_H})\cap\mathcal{S}_{c,\{1,...,h\}}(\Omega_{D_H})$;
    \item if $f\in\mathcal{S}_{c,h}(\Omega_D)$, $f^\circ_{s,h}=f$ and $f'_{s,h}=0$;
    \item if $h\in H$, $H\cap\{1,...,h-1\}\neq\emptyset$ and $f\in\mathcal{S}_h(\Omega_D)$, then $f'_{s,H}=0$;
    \item $(f^\circ_{s,h})^\circ_{s,h}=f^\circ_{s,h}$ and $(f'_{s,h})'_{s,h}=0$.
\end{enumerate}
\end{prop}

\begin{proof}
\begin{enumerate}
    \item  If $f=\mathcal{I}(F)$, by definition $f^\circ_{s,h}=\sum_{K\subset H^c}[J_K,F_K]$, hence by Proposition \ref{proposition characterization partial circularity}, $f^\circ_{s,H}\in\mathcal{S}_{c,H}(\Omega_D)$. Moreover, we can write it as 
    \begin{equation*}
    f^\circ_{s,h}=\sum_{K\subset (H\cup p)^c}[J_K,F_K]+J_p\sum_{K\subset (H\cup p)^c}[J_K,F_{K\cup p}],
\end{equation*}
   so $f^\circ_{s,h}\in\mathcal{S}_p(\Omega_D)$. In the same way one can prove that $f'_{s,H}\in\mathcal{S}_{c,H}(\Omega_{D_H})\cap\mathcal{S}_p(\Omega_{D_H})$.
    \item By Proposition \ref{characterization sliceness in xh}, $F$ takes the form
\begin{equation*}
    F= \sum_{K\in\mathcal{P}(n),h\notin K}e_KF_K+e_{\{h\}} \sum_{Q\subset\{h+1,...,n\}}e_QF_{\{h\}\cup Q},
\end{equation*}
hence,
\begin{equation*}
    F'_h=\beta_h^{-1} \sum_{Q\subset\{h+1,...,n\}}e_QF_{\{h\}\cup Q}.
\end{equation*}
This shows that $f'_{s, h}\in\mathcal{S}_{c,\{1,...,h\}}(\Omega_{D_h})$, by Proposition \ref{proposition characterization partial circularity}. Finally, by Proposition \ref{characterization sliceness in xh}, $f'_{s,h}\in\mathcal{S}_{h+1}(\Omega_{D_h})$.
\item By Proposition \ref{proposition characterization partial circularity}, $F=\sum_{K\in\mathcal{P}(n),h\notin K}e_KF_K$, so $F'_h=0$ and $F^\circ_h=F$.
\item Let $i\in H\cap\{1,...,h-1\}\neq\emptyset$, since $f\in\mathcal{S}_h(\Omega_D)$, by (2) $f'_{s,h}\in\mathcal{S}_{c,i}(\Omega_{D_i})$ and by (3) $(f'_{s,h})'_{s,i}=0$. In particular, $f'_{s,H}=0$.
\item It follows from (1) and (3).
\end{enumerate}
\end{proof}

Partial spherical derivatives do not affect regularity in other variables.

\begin{prop}
Let $f\in\mathcal{S}^1(\Omega_D)$. Suppose that $f\in\ker(\partial/\partial x_t^c)$ for some $t=1,...,n$, then $f'_{s, h}\in\ker(\partial/\partial x_t^c)$, $\forall h\neq t$.
\end{prop}

\begin{proof}
Let $f=\mathcal{I}(F)$, with $F=\sum_{K\in\mathcal{P}(n)}e_KF_K$, so $f'_{s, h}=\mathcal{I}(F'_h)$, with $F'_h=\sum_{K\in\mathcal{P}(n)}e_KG_K$, $G_K=0$, if $h\in K$ and $G_K=\beta_h^{-1}F_{K\cup\{h\}}$, if $h\notin K$. Let $K\in\mathcal{P}(n)$, with $h,t\notin K$, then by the regularity of $F$ it holds
\begin{equation*}
    \left\{\begin{array}{l}
         \dfrac{\partial G_K}{\partial\alpha_t}=\dfrac{\partial \beta_h^{-1}F_{K\cup\{h\}}}{\partial\alpha_t}=\beta_h^{-1}\dfrac{\partial F_{K\cup\{h\}}}{\partial\alpha_t}=\beta_h^{-1}\dfrac{\partial F_{K\cup\{h\}\cup\{t\}}}{\partial\beta_t}=\dfrac{\partial G_{K\cup\{t\}}}{\partial\beta_t}\\
        \dfrac{\partial G_K}{\partial\beta_t}=\dfrac{\partial \beta_h^{-1}F_{K\cup\{h\}}}{\partial\beta_t}=\beta_h^{-1}\dfrac{\partial F_{K\cup\{h\}}}{\partial\beta_t}=-\beta_h^{-1}\dfrac{\partial F_{K\cup\{h\}\cup\{t\}}}{\partial\alpha_t}=-\dfrac{\partial G_{K\cup\{t\}}}{\partial\alpha_t}.
    \end{array}\right.
\end{equation*}
This proves that $F'_h$ is $t$-holomorphic, hence $f'_{s,h}\in\ker(\partial/\partial x_t^c)$.
\end{proof}

As recalled in \S \ref{subsection one variable theory}, every one variable slice function $f$ can be decomposed as $f(x)=f^\circ_s(x)+\operatorname{Im}(x)f'_s(x)$. We now give a similar decomposition for every variable, through the slice product.

\begin{prop}
    Let $f\in\mathcal{S}(\Omega_D)$, then for any $h=1,...,n$ we can decompose
    \begin{equation}
    \label{representation formula several}
        f=f^\circ_{s, h}+\operatorname{Im}(x_h)\odot f'_{s, h}.
    \end{equation}
\end{prop}
\begin{proof}
    Let $f=\mathcal{I}(F)$, with $F=\sum_{K\in\mathcal{P}(n)}e_KF_K$. Suppose first $x\in\mathbb{R}_h$, i.e. $\operatorname{Im}(x_h)(x)=0$, then by \eqref{defining property of stem}, with the usual notation, we have
    \begin{equation*}
        f(x)=\sum_{K\in\mathcal{P}(n)}[J_K,F_K(z)]=\sum_{K\in\mathcal{P}(n), h\notin K}[J_K,F_K(z)]=f^\circ_{s,h}(x).
    \end{equation*}
    Now, suppose $x\in\Omega_D\setminus\mathbb{R}_h$ and define $\operatorname{Im}(Z_h)(z_1,...,z_n):=e_h\beta_h$, where $z_j:=\alpha_{j}+i\beta_j$. $\operatorname{Im}(Z_h)\in Stem(D)$ and $\mathcal{I}(\operatorname{Im}(Z_h))=\operatorname{Im}(x_h)$. Then
    \begin{equation*}
\begin{split}
            F^\circ_{h}+\operatorname{Im}(Z_h)\otimes F'_h&= \sum_{K\in\mathcal{P}(n),h\notin K}e_KF_K+(e_h\beta_h)\otimes\left( \sum_{K\in\mathcal{P}(n),h\notin K}e_K\beta_h^{-1}F_{K\cup\{h\}}\right)\\
            &= \sum_{K\in\mathcal{P}(n),h\notin K}e_KF_K+ \sum_{K\in\mathcal{P}(n),h\notin K}e_{K\cup\{h\}}F_{K\cup\{h\}}=F.
\end{split}
    \end{equation*}
    Finally, $f=\mathcal{I}(F)=\mathcal{I}(F^\circ_{h}+\operatorname{Im}(Z_h)\otimes F'_h)=f^\circ_{s,h}+\operatorname{Im}(x_h)\odot f'_{s,h}$.
\end{proof}

Next proposition shows that the partial spherical derivatives satisfies a Leibniz-type formula, analogue to the one-dimensional case.
\begin{prop}[Leibniz rule]
\label{Proposizione Leibniz formula}
    Let $f,g\in\mathcal{S}(\Omega_D)$. It holds
    \begin{equation}
    \label{Leibniz rule}
        (f\odot g)'_{s, h}=f'_{s, h}\odot g^\circ_{s, h}+f^\circ_{s, h}\odot g'_{s, h}.
    \end{equation}
\end{prop}

\begin{proof}
    Let $f=\mathcal{I}(F)$ and $g=\mathcal{I}(G)$. We have to show that $(F\otimes G)'_h=F'_h\otimes G^\circ_{h}+F^\circ_{h}\otimes G'_h$.
    By \cite[Lemma 2.34]{Several} we have $F'_h\otimes G^\circ_{h}= \sum_{K\in\mathcal{P}(n), h\notin K}e_K(F'_h\otimes G^\circ_{h})_K$, where
    \begin{equation*}
        (F'_h\otimes G^\circ_{h})_K= \sum_{K_1,K_2,K_3\in\mathcal{D}(K)}(-1)^{|K_3|}(F'_h)_{K_1\cup K_3}(G^\circ_{h})_{K_2\cup K_3},
    \end{equation*}
    and $\mathcal{D}(K):=\{(K_1,K_2,K_3)\in\mathcal{P}(n)^3\mid K=K_1\sqcup K_2, K_3\cap K=\emptyset\}.$ By definition of $F'_h$ and $G^\circ_{h}$, the previous equation reduces to
     \begin{equation*}
        (F'_h\otimes G^\circ_{h})_K= \sum_{K_1,K_2,K_3\in\mathcal{D}'_h(K)}(-1)^{|K_3|}F_{K_1\cup K_3\cup\{h\}}G_{K_2\cup K_3},
    \end{equation*}
    with $\mathcal{D}'_h(K):=\{(K_1,K_2,K_3)\in\mathcal{P}(n)^3\mid K=K_1\sqcup K_2, K_3\cap (K\cup\{h\})=\emptyset\}.$ In the very same way, we get 
    \begin{equation*}
        (F^\circ_{h}\otimes G'_h)_K= \sum_{K_1,K_2,K_3\in\mathcal{D}'_h(K)}(-1)^{|K_3|}F_{K_1\cup K_3}G_{K_2\cup K_3\cup\{h\}},
    \end{equation*}
   hence
       \begin{equation*}
        F'_h\otimes G^\circ_{h}+F^\circ_{h}\otimes G'_h= \sum_{K\in\mathcal{P}(n), h\notin K}e_K \sum_{K_1,K_2,K_3\in\mathcal{D}'_h(K)}(-1)^{|K_3|}\left(F_{K_1\cup K_3\cup\{h\}}G_{K_2\cup K_3}+F_{K_1\cup K_3}G_{K_2\cup K_3\cup\{h\}}\right).
    \end{equation*}
    On the other hand, $F\otimes G=\sum_{K\in\mathcal{P}(n)}e_K(F\otimes G)_K$, where
    \begin{equation*}
        \begin{split}
            (F\otimes G)_K= \sum_{K_1,K_2,K_3\in\mathcal{D}(K)}(-1)^{|K_3|}F_{K_1\cup K_3}G_{K_2\cup K_3}.
        \end{split}
    \end{equation*}
    Thus 
    \begin{align*}
        (F\otimes G)'_h&= \sum_{K\in\mathcal{P}(n), h\notin K}e_K\beta_h^{-1}(F\otimes G)_{K\cup\{h\}}\\
        &= \sum_{K\in\mathcal{P}(n), h\notin K}e_K \sum_{K_1,K_2,K_3\in\mathcal{D}(K\cup\{h\})}(-1)^{|K_3|}F_{K_1\cup K_3}G_{K_2\cup K_3}.
    \end{align*}
    Note that
    \begin{equation*}
        \begin{split}
            \mathcal{D}(K\cup\{h\})&=\{(K_1,K_2,K_3)\in\mathcal{P}(n)^3\mid K\cup\{h\}=K_1\sqcup K_2, K_3\cap(K\cup\{h\})=\emptyset\}\\
            &=\{(K_1\cup\{h\},K_2,K_3),(K_1,K_2\cup\{h\},K_3)\mid (K_1,K_2,K_3)\in\mathcal{D}'_h(K)\}
        \end{split}
    \end{equation*}
    so
    \begin{equation*}
        \begin{split}
            (F\otimes G)'_h&= \sum_{K\in\mathcal{P}(n), h\notin K}e_K \sum_{K_1,K_2,K_3\in\mathcal{D}(K\cup\{h\})}(-1)^{|K_3|}F_{K_1\cup K_3}G_{K_2\cup K_3}\\
            &= \sum_{K\in\mathcal{P}(n), h\notin K}e_K \sum_{K_1,K_2,K_3\in\mathcal{D}'_h(K)}(-1)^{|K_3|}\left(F_{K_1\cup\{h\}\cup K_3}G_{K_2\cup K_3}+F_{K_1\cup K_3}G_{K_2\cup\{h\}\cup K_3}\right)\\
            &=F'_h\otimes G^\circ_{h}+F^\circ_{h}\otimes G'_h.
        \end{split}
    \end{equation*}
\end{proof}

\begin{cor}
    Let $f\in\mathcal{S}(\Omega_D)$ and $g\in\mathcal{S}_{c,H}(\Omega_D)$ for some $H\in\mathcal{P}(n)$, then $(f\odot g)'_{s,H}=f'_{s,H}\odot g$.
\end{cor}

\begin{proof}
    We proceed by induction over $|H|$. Suppose first $|H|=1$, then it follows from Proposition \ref{Proposizione Leibniz formula} and Proposition \ref{proposizione proprieta derivata sferica} (3). Now, suppose by induction that $(f\odot g)'_{s,H}=f'_{s,H}\odot g$ and let $h\notin H$, then in the same way we have
    \begin{equation*}
        (f\odot g)'_{s,H\cup\{h\}}=(f'_{s,h}\odot g^\circ_{s,h}+f^\circ_{s,h}\odot g'_{s,h})'_{s,H}=(f'_{s,h}\odot g)'_{s,H}=f'_{s,H\cup\{h\}}\odot g.
    \end{equation*}
\end{proof}

The next result highlights a fundamental property of  partial spherical derivatives, i.e. harmonicity. The only requirement is regularity in such variable. This extends the result for one-variable slice regular functions \cite[Theorem 6.3, (c)]{Harmonicity}.
\begin{prop}
\label{proposizione derivata sferica e armonica}
Let $f\in\mathcal{S}^1(\Omega_D)$. Suppose that $f\in\ker(\partial/\partial x_h^c)$, for some $h=1,...,n$. Then 
\begin{equation*}
    \Delta_hf'_{s, h}=0.
\end{equation*}
\end{prop}

\begin{proof}
Let us introduce a slightly different notation: let $x=(x_1,...,x_n)\in\Omega_D$, with $x_l=\alpha_l+i\beta_l+j\gamma_l+k\delta_l=\alpha_l+J_lb_l$, where 
\begin{equation*}
    J_l:=\dfrac{i\beta_l+j\gamma_l+k\delta_l}{\sqrt{\beta^2_l+\gamma^2_l+\delta^2_l}},\qquad b_l(\beta_l,\gamma_l,\delta_l):=\sqrt{\beta^2_l+\gamma^2_l+\delta^2_l}.
\end{equation*}
    Let $f=\mathcal{I}(F)$, with $F=\sum_{K\in\mathcal{P}(n)}e_KF_K$, then, by definition,
$        f'_{s, h}(x)= \sum_{K\in\mathcal{P}(n),h\notin K}\linebreak[4] [J_K,b_h^{-1}F_{K\cup\{h\}}(z)]
$ and so $$\Delta_hf'_{s, h}= \sum_{K\in\mathcal{P}(n),h\notin K}[J_K,\Delta_h(b_h^{-1}F_{K\cup\{h\}})].$$
Thus, it is enough to prove that 
\begin{equation*}
    \Delta_h(b_h^{-1}F_{K\cup\{h\}})=\left(\partial_{\alpha_h}^2+\partial_{\beta_h}^2+\partial_{\gamma_h}^2+\partial_{\delta_h}^2\right)\left(b_h^{-1}F_{K\cup\{h\}}(z',\alpha_h+ib_h(\beta_h,\gamma_h,\delta_h),z")\right)=0.
\end{equation*}
Immediately we get $\partial_{\alpha_h}^2(b_h^{-1}F_{K\cup\{h\}})=b_h^{-1}\partial_{\alpha_h}^2F_{K\cup\{h\}}$. Moreover, by $$\partial_{\beta_h}b_h=\beta_h/b_h,\qquad\partial_{\beta_h}F_{K\cup\{h\}}=\dfrac{\beta_h}{b_h}\partial_{b_h}F_{K\cup\{h\}},$$
we find
\begin{equation*}
    \partial_{\beta_h} (b_h^{-1}F_{K\cup\{h\}})=-\dfrac{\beta_h}{b_h^3}F_{K\cup\{h\}}+\dfrac{\beta_h}{b_h^2}\partial_{b_h}F_{K\cup\{h\}}
\end{equation*}
and
\begin{equation*}
    \begin{split}
        \partial^2_{\beta_h}(b_h^{-1}F_{K\cup\{h\}})&=\partial_{\beta_h}\left(-\dfrac{\beta_h}{b_h^3}F_{K\cup\{h\}}+\dfrac{\beta_h}{b_h^2}\partial_{b_h}F_{K\cup\{h\}}\right)\\
        &=\dfrac{3\beta_h^2-b_h^2}{b_h^5}F_{K\cup\{h\}}-\dfrac{\beta_h^2}{b_h^4}\partial_{b_h}F_{K\cup\{h\}}+\dfrac{b_h^2-2\beta_h^2}{b_h^4}\partial_{b_h}F_{K\cup\{h\}}+\dfrac{\beta_h^2}{b_h^3}\partial_{b_h}^2F_{K\cup\{h\}}\\
        &=\dfrac{3\beta_h^2-b_h^2}{b_h^5}F_{K\cup\{h\}}+\dfrac{b_h^2-3\beta_h^2}{b_h^4}\partial_{b_h}F_{K\cup\{h\}}+\dfrac{\beta_h^2}{b_h^3}\partial_{b_h}^2F_{K\cup\{h\}}.
    \end{split}
\end{equation*}
Analogously for $\gamma_h$ and $\delta_h$:
\begin{equation*}
    \partial^2_{\gamma_h}(b_h^{-1}F_{K\cup\{h\}})=\dfrac{3\gamma_h^2-b_h^2}{b_h^5}F_{K\cup\{h\}}+\dfrac{b_h^2-3\gamma_h^2}{b_h^4}\partial_{b_h}F_{K\cup\{h\}}+\dfrac{\gamma_h^2}{b_h^3}\partial_{b_h}^2F_{K\cup\{h\}},
\end{equation*}
\begin{equation*}
    \partial^2_{\delta_h}(b_h^{-1}F_{K\cup\{h\}})=\dfrac{3\delta_h^2-b_h^2}{b_h^5}F_{K\cup\{h\}}+\dfrac{b_h^2-3\delta_h^2}{b_h^4}\partial_{b_h}F_{K\cup\{h\}}+\dfrac{\delta_h^2}{b_h^3}\partial_{b_h}^2F_{K\cup\{h\}}.
\end{equation*}
So
\begin{equation*}
    \begin{split}
\left(\partial^2_{\beta_h}+\partial^2_{\gamma_h}+\partial^2_{\delta_h}\right)(b_h^{-1}F_{K\cup\{h\}})&=b_h^{-1}\partial^2_{b_h}F_{K\cup\{h\}}
    \end{split}
\end{equation*}
and finally
\begin{equation*}
    \Delta_h(b_h^{-1}F_{K\cup\{h\}})=b_h^{-1}\left(\partial_{\alpha_h}^2+\partial_{b_h}^2\right)F_{K\cup\{h\}}=0,
\end{equation*}
since $f\in\ker(\partial/\partial x_h^c)$, which implies the $h$-holomophicity of every $F_K$.
\end{proof}

Our last application is a generalization to several variables of Fueter's Theorem, which is a fundamental result in hypercomplex analysis. In modern language, it states that, given a slice regular function $f:\Omega_D\subset\mathbb{H}\rightarrow\mathbb{H}$, its Laplacian generates an axially monogenic function, i.e.
\begin{equation*}
    \overline{\partial}_{CRF}\Delta f=0.
\end{equation*}


\begin{thm}
\label{Teorema di Fueter in piu variabili}
Let $\Omega_D\subset\mathbb{H}^n$ be a circular set and let $f\in\mathcal{S}\mathcal{R}_h(\Omega_D)$ be a slice function, which is slice regular w.r.t. $x_h$, for some $h=1,...,n$. Then $\Delta_h f$ is an axially monogenic function w.r.t. $x_h$, i.e. 
\begin{equation*}
    \Delta_hf\in\ker(\overline{\partial}_{x_h}).
\end{equation*}
In other words, the Fueter map extends to $$\Delta_h:\mathcal{S}\mathcal{R}_h(\Omega_D)\rightarrow\mathcal{A}\mathcal{M}_h(\Omega_D).$$
\end{thm}

\begin{proof}
    Since $f\in\mathcal{S}\mathcal{R}_h(\Omega_D)$, we can apply Lemma \ref{lemma derivata sferica operatore crf e Laplaciano} 1. and Proposition \ref{proposizione derivata sferica e armonica}
    \begin{equation*}
        \overline{\partial}_{x_h}\Delta_hf=\Delta_h\overline{\partial}_{x_h}f=-2\Delta_hf'_{s,h}=0.
    \end{equation*}
\end{proof}

\printbibliography

\end{document}